\documentclass{article}
\usepackage{amsthm}
\usepackage{amsmath}
\usepackage{amsfonts}
\usepackage[letterpaper,body={16.5cm,24cm}, mag=3000]{geometry}
\usepackage{amssymb}
\usepackage{secdot}
\usepackage{setspace}
\usepackage{titlesec}
\titleformat{\section}[runin]{\bfseries\filcenter}{\thesection}{1em}{}

\renewcommand{\thesection}{\arabic{section}}
\title{\large \bf On the conjecture of non-inner automorphisms of  finite $p$-groups with a non-trivial abelian direct factor} 
\author{\small \bf Mandeep Singh$^\ast$ and Mahak Sharma$^{\ast\ast}$\\
	\small \em $^\ast$Department of Mathematics, Arya College, Ludhiana - 141 001,
	India\\
	\small \em $^{\ast\ast}$Department of Mathematics, Goswami Ganesh Dutta Sanatan Dharma College, Chandigarh - 160019,
	India\\
}
\date{}

\newtheorem{thm}{Theorem}[section]
\newtheorem{lm}[thm]{Lemma}
\newtheorem{cor}[thm]{Corollary}

\begin{document}
	\maketitle
	
	\begin{abstract}
		\noindent Let $p$ be a prime number. A longstanding conjecture asserts that every finite non-abelian $p$-group has a non-inner automorphism of order $p$. In this paper, we prove that the conjecture is true when a finite non-abelian $p$-group $G$ has a non-trivial abelian direct factor. Moreover, we prove that the non-inner automorphism is central and fixes $\Phi(G)$ elementwise. As a consequence, we prove that every group which is not purely non-abelian has a non-inner central automorphism of order $p$ which fixes $\Phi(G)$ elementwise.
	\end{abstract}

	\vspace{2ex}
	
	\vspace{2ex}


	\section{Introduction}
	
	Let $G$ be a finite non-abelian $p$-group, where $p$ is a prime number. Let $G'$, $Z(G)$,  $\Phi(G)$ and $d(G)$ respectively denote the derived subgroup, the center, the Frattini subgroup and the minimal number of generators of $G$.  Let $\Omega_1(G)$ be the subgroup of $G$ generated by all elements of order $p$ and let $G^p$ denote the group generated by the set $\lbrace g^p : g \in G \rbrace$. For $p=2$, $\Phi(G)=G^{2}$ and for $p>2$, $\Phi(G)=G^{\prime}G^p$. In 1973, Berkovich \cite{ber2008} proposed the following conjecture (see for eg. \cite[Problem 4.13]{mazkhu}).
	
	\begin{center}
		{\bf Prove that every finite non-abelian $p$-group admits an automorphism of order $p$ which is not an inner one.}
	\end{center}
	
	This conjecture has been settled for the following classes of $p$-groups:
	$G$ is regular \cite{deasil2002, sch1980};
	$G/Z(G)$ is powerful \cite{abd2010}; $C_G(Z(\Phi(G)))\neq \Phi(G)$ \cite{deasil2002}; the nilpotency class of $G$ is $2$ or $3$ \cite{abd2007, lie1965}; $G$ is a $p$-group, where $p >2$ such that $(G, Z(G))$ is a Camina pair \cite{gho2013}, $Z(G)$ is non-cyclic and $W(G)$ is non-abelian, where $W(G)/Z(G)=\Omega_1(Z_2(G)/Z(G))$ \cite{garsin2021}.
	Most recently, in 2024, Komma \cite{kom2024} proved that if $G$ is of coclass $4$ and $5$, where $p \geq 5$ or $G$ is an odd order $p$-group with cyclic center satisfying $C_G(G^p \gamma_3(G)) \cap Z_3(G)\leq Z(\Phi(G))$, then the conjecture holds. For more detail on this conjecture, the reader are advised to see (\cite{abdgho2014}, \cite{att2013}, \cite{gho2014}, \cite{kom2024}, \cite{ruslegyad2016},\cite{msingh2025}) and references therein.
	
	In most of the cases, it is observed that if the conjecture does not hold for $G$, then the center $Z(G)$ is cyclic. The following natural question was asked in \cite{garsin2021}.\\

	Given a finite non-abelian $p$-group $G$ with non-cyclic center $Z(G)$, under what conditions does the conjecture hold?\\
	
	If we consider $Z(G)$ to be non-cyclic, then rank of $Z(G)$ must be atleast 2. While working on this objective, we prove that if $H$ and $K$ are normal subgroups of $G$ such that $G = H \times K$, where $H$ is non-trivial abelian and $K$ is non-abelian, then $G$ has a  non-inner central automorphism of order $p$ which fixes $\Phi(G)$ elementwise and as a consequence, we confirm the validity of the non-inner conjecture for another class of finite non-abelian $p$-groups i.e. all those groups which are not purely non-abelian.

	\section{Main results}
	An automorphism $\alpha$ of $G$ is called central if $g^{-1} \alpha(g) \in Z(G)$ for all $g \in G$. Let $H$ and $K$ be normal subgroups of a group $G$, then $G$ is called the interal direct product of $H$ and $K$ denoted by $G = H \times K$ if $G=HK$ and $H \cap K=1$.\\
	
	The following Lemma is well known.

	\begin{lm}
		Let $G$ be a group and let $x, y, z \in G$, then 
		\begin{enumerate}
			\item[$(i)$] $[x,yz]=[x,z][x,y][x,y,z]$.
			\item[$(ii)$] $[xy,z]=[x,z][x,z,y][y,z]$.
		\end{enumerate}
	\end{lm}

	\begin{thm}
		Let $G$ be a finite non-abelian $p$-group 
		and let $H$ and $K$ be normal subgroups of $G$ such that $G = H \times K$, where $H$ is non-trivial abelian and $K$ is non-abelian. Then $G$ has a  non-inner central automorphism of order $p$ which fixes $\Phi(G)$ elementwise. 
	\end{thm}
	
	\begin{proof}
		Let $M$ be a maximal subgroup of $H$ and $h \in H \setminus M$. Then $H=M<h>$ and therefore $G = M<h> \times K$. Thus every element of $G$ can be written in form $mh^ik$, where $m \in M$, $k \in K$ and $i$ is a positive integer. Let $g \in \Omega_1(Z(K))$. Define a map $\alpha$ on $G$ by $$\alpha(mh^ik)=mh^ikg^i \;\;\text{for some} \;\;i \in \lbrace 0,1,2,...p-1 \rbrace.$$ Note that $\alpha$ is a central automorphism of $G$ as 
		$$(mh^ik)^{-1} \alpha(mh^ik)=g^i \in Z(K)\simeq 1 \times Z(K) \subseteq Z(H) \times Z(K)=Z(G).$$
		
		We claim that $\alpha$ is a non-inner automorphism of $G$.
		On the contrary, let us suppose that $\alpha$ is an inner automorphism. Let $\alpha=f_{m_1h^jk_1}$, where $m_1 \in M$, $k_1 \in K$ and $j$ is a positive integer. Then $f_{m_1h^jk_1}(h)=hg$ which implies that $hg=(m_1h^jk_1)^{-1}h(m_1h^jk_1)=h$ and therefore $g=1$. This contradicts the choice of $g$.\\ 
		Observe that for $i=0$, $\alpha$ fixes   $MK$ elementwise, therefore it fixes $K$ elementwise because $K \subseteq MK$. Note that, $\alpha^p(mh^ik)=mh^ikg^{pi}=mh^ik(g^p)^i=mh^ik$ because order of $g$ is $p$. Therefore order of $\alpha$ is $p$. 
		
		Next we prove that $\alpha$ fixes $G^{\prime}$ elementwise. Let $[x,y] \in G^{\prime}$, where $x, y \in G$. Since $\alpha$ is a central automorphism, let $\alpha(x)=xz_1$ and $\alpha(y)=yz_2$ for some $z_1, z_2 \in Z(G)$. Then by using the Lemma 2.1, we have,    
		$$\alpha([x,y])=[\alpha(x), \alpha(y)]=[xz_1,yz_2]=[x,y].$$
		
		Since $(mh^ik)^{-1} \alpha(mh^ik)$ is a central element of order $p$, $$\alpha((mh^ik)^p)= (mh^ik)^p.$$ Therefore $\alpha$ fixes $G^p$ elementwise. Hence $\alpha$ fixes $\Phi(G)$ elementwise.
	\end{proof}

	\begin{cor}
		Let $G$ be a finite non-abelian $p$-group 
		and let $H_1$, $H_2$, $H_3$,...,$H_n$ be normal subgroups of $G$ such that $G = H_1 \times H_2 \times H_3 \times...\times H_n$, where $H_1$ is abelian and $H_i, i \neq 1$ are non-abelian. Then $G$ has a non-inner central automorphism of order $p$ which fixes $\Phi(G)$ elementwise.
		
	\end{cor}

	\begin{proof}
		For $n=2$, the result follows from Theorem 2.3 by taking $H_1=H$ and $H_2=K$. For $n=3$, Let $K=H_2 \times H_3$. Then $K = H_2H_3$ and $H_2 \cap H_3 = 1$. Since $H_2$, $H_3$ be normal subgroups of $G$, $H_2H_3$ is also a normal subgroup of $G$. Therefore $G=H_1 \times K$. Thus by Theorem 2.3, $G$ has a non-inner central automorphism of order $p$ which fixes $\Phi(G)$ elementwise. By continuing in this manner, the result follows.
		
	\end{proof}
	
	A finite non-abelian group $G$ is called purely non-abelian if it has no non-trivial abelian direct factor. 
	The following corollary is an immediate consequence of the above theorem.
	\begin{cor}
		Let $G$ be a finite non-abelian $p$-group 
		such that $G$ is not purely non-abelian. Then $G$ has a non-inner central automorphism of order $p$ which fixes $\Phi(G)$ elementwise. 
		
	\end{cor}

	We conclude this paper with an example of a group that satisfies the hypothesis of  Theorem 2.2. Consider
	$$G=(C_3 \times C_9) \times ((C_9 \times C_3)\rtimes C_3)=\langle x, y, z, a, b, c, d \rangle,$$
	
	with the relations:
	$a^3=b^3=c^3=d^3=1, [x,z]=[y,z]=[y,a]=[z,a]=[x,a]=[z,b]=
	[a,b]=[y,b]=[x,d]=[b,d^{-1}]=1, y^2dy=x^2c^{-1}x=z^2c^{-1}z=1, xbyx^{-1}y^{-1}=dxbx^{-1}b^{-1}=1$.
	This group is the group number 8028 in the GAP Library of groups of order $3^7$. Note that $Z(G)=\langle a, c, d, z \rangle$. 
	There are 76527504 automorphisms of this group. One of these automorphisms is $\alpha$, where 
	\begin{eqnarray*}
		&
		\alpha(x^2c^2)=x^2, \; \alpha(xb^2c^2d)=xb^2cd,\;  \alpha(xyza^2b^2c^2d)=xyzab^2d,  
		&
		\\
		&
		\alpha(xyzabc^2)=xyzbc, \;\alpha(x^2y^2a^2b^2cd)=x^2y^2a^2b^2d.
	\end{eqnarray*}

	It is not very difficult to see that by using the relators of the group $G$, the map $\alpha$ can be reduced to the following form:
	$$
	\alpha (x)=xc^2, ~ \alpha (z)=za^2c, ~ \alpha (a)=ac^2 ~ \mbox{and} ~ \alpha(g)=g,  \mbox{for all} ~ g \in \lbrace y, b, c, d \rbrace.
	$$

	Since $\alpha$ maps $a$ to $ac^2$, $\alpha$ is a non-inner automorphism of $G$.
	Also $\alpha^3(a)=ac^6=a$ and $\alpha^3(g)=g$ $ \forall \; g \in \lbrace x, y, z, b, c, d \rbrace$. Therefore order of $\alpha$ is 3. It is easy to check that $\alpha$ is a central automorphism of $G$. As proved in Theorem 2.2, by using the similar arguments, one can show that $\alpha$ fixes  $G^{\prime}$ elementwise. Note that $$G^3=\langle x^3, y^3, z^3, a^3, b^3, c^3, d^3 \rangle = \langle x^3, y^3, z^3 \rangle.$$
	Observe that since $c \in Z(G)$, $\alpha(x^3)= x^3c^6=x^3$. Therefore $\alpha$ fixes $x^3$. Also since $\alpha$ fixes $y$, it fixes $y^3$. Now $\alpha(z^3)=z^3a^6c^3=z^3$. Thus $\alpha$ fixes $G^3$. Hence $\alpha$ fixes $\Phi(G)=G^{\prime}G^3$ elementwise.

\end{document}